\documentclass[letterpaper, reqno, 11pt]{amsart}
\usepackage{amssymb, amscd, latexsym, amsfonts, mathrsfs}
\usepackage[english]{babel}
\RequirePackage[utf8]{inputenc}
\usepackage{enumerate, cite, calc, xstring}
\usepackage{float, wrapfig, sidecap, multicol, needspace}

\usepackage[T1]{fontenc}
\usepackage[all]{xy}
\usepackage{graphicx, color, url}
\usepackage[top=3cm, bottom=3cm, left=3cm, right=3cm]{geometry}

\newtheorem{theorem}{Theorem}[section]

\newtheorem{corollary}[theorem]{Corollary}
\newtheorem{lemma}[theorem]{Lemma}

\theoremstyle{definition}

\theoremstyle{remark}

\numberwithin{equation}{section}

\begin{document}

\title[Geometry of Curves in $\mathbb R^n$, SVD, and Hankel Determinants]{Geometry of Curves in $\mathbb R^n$, Singular Value Decomposition, and Hankel Determinants}

\author[J. \'Alvarez-Vizoso]{Javier \'Alvarez-Vizoso}
\author[R. Arn]{Robert Arn}
\author[M. Kirby]{Michael Kirby}
\author[C. Peterson]{Chris Peterson}
\address{Department of Mathematics, Colorado State University, Fort Collins, CO, USA}
\email{alvarez@math.colostate.edu, arn@math.colostate.edu, kirby@math.colostate.edu,\newline peterson@math.colostate.edu}

\author[B. Draper]{Bruce Draper}
\address{Department of Computer Science, Colorado State University, Fort Collins, CO, USA}
\email{draper@cs.colostate.edu}

\date{\today}

\maketitle

% REQUIRED
\begin{abstract}
Let $\gamma: I \rightarrow \mathbb R^n$ be a parametric curve of class $C^{n+1}$, regular of order $n$. The Frenet-Serret apparatus of $\gamma$ at $\gamma(t)$ consists of a frame $e_1(t), \dots , e_n(t)$ and generalized curvature values $\kappa_1(t), \dots, \kappa_{n-1}(t)$. Associated with each point of $\gamma$ there are also local singular vectors $u_1(t), \dots, u_n(t)$ and local singular values $\sigma_1(t), \dots, \sigma_{n}(t)$. This local information is obtained by considering a limit, as $\epsilon$ goes to zero, of covariance matrices defined along $\gamma$ within an $\epsilon$-ball centered at  $\gamma(t)$. We prove that for each $t\in I$,  the Frenet-Serret frame and the local singular vectors agree at $\gamma(t)$ and that the values of the curvature functions at $t$ can be expressed as a fixed multiple of a ratio of local singular values at $t$. More precisely, we show that if $\gamma(t)\subset \mathbb R^n$ for any $n\in\mathbb N$ then, for each $i$ between $2$ and $n$, $\kappa_{i-1}(t)=\sqrt{a_{i-1}}\frac{\sigma_{i}(t)}{\sigma_1(t) \sigma_{i-1}(t)}$ with $a_{i-1} = \left(\frac{i}{i+(-1)^i}\right)^2 {\frac{4i^2-1}{3}}$. For this we prove a general formula for the recursion relation of a certain class of sequences of Hankel determinants using the theory of monic orthogonal polynomials and moment sequences.
\end{abstract}

%%%%%%%%%%%%%%%%%%%%%%%%%%%%%%%%%%%%%%%%%%%%%%%%%%%%%%%%%%%%%%%%%%%%%%
%%%%%%%%%%%%%%%        MAIN TEXT
%%%%%%%%%%%%%%%%%%%%%%%%%%%%%%%%%%%%%%%%%%%%%%%%%%%%%%%%%%%%%%%%%%%%%%

\section{Introduction}
\label{intro}
Consider an interval $I\subset \mathbb R$ and a vector valued function $\gamma: I \rightarrow \mathbb R^n$. If $\gamma$ is $k$ times differentiable, with continuous derivatives, then $\gamma$ is said to be a parametric curve of class $C^k$. Let $\gamma^{(k)}$ denote the $k^{th}$ derivative of $\gamma$. If for each $t\in I$, the set of vectors $\{\gamma^{(1)}(t), \gamma^{(2)}(t), \dots, \gamma^{(r)}(t)\}$ are linearly independent in $\mathbb R^n$, then $\gamma$ is said to be regular of order $r$. If $\|\gamma^\prime(t)\|=1$ for each $t\in I$ then $\gamma$ is said to be parameterized by arc length. 

Let $\gamma: I \rightarrow \mathbb R^n$ be a parametric curve of class $C^{n+1}$, regular of order $n$, parameterized by arc length. 
At any point $\gamma(t)\in \gamma(I)$, the Frenet-Serret frame
is determined by applying the Gram-Schmidt process to the
vectors $\gamma^{(1)}(t), \gamma^{(2)}(t), \dots, \gamma^{(n)}(t)$. Thus the Frenet-Serret frame at $\gamma(t)$ is the ordered sequence of orthonormal vectors\newline $e_1(t), e_2(t), \dots, e_n(t)$, where 
$$e_i(t)=\frac{\tilde{e}_i(t)}{\|\tilde{e}_i(t)\|} \hskip .1in {\rm with} \hskip .1in \tilde{e}_{i}(t) = \gamma^{(i)}(t)-\sum_{k=1}^{i-1}<\gamma^{(i)}(t),e_k(t)>e_k(t) \hskip .1in {\rm for} \hskip .1in 1\leq i \leq n.$$
The generalized curvature functions  of $\gamma$ are defined by $$\kappa_i(t) =\ <e_i^\prime(t),e_{i+1}(t)> \hskip .1in {\rm for} \hskip .1in 1\leq i \leq n-1.$$ With this definition, $\kappa_i(t) >0$ for all $i$.
The frame functions $e_1(t), e_2(t), \dots, e_n(t)$ together with the generalized curvature functions $\kappa_1(t), \dots, \kappa_{n-1}(t)$ is called the Frenet-Serret apparatus of $\gamma$. The Frenet-Serret apparatus of a curve characterizes the curve up to translation.

By the definition of the $e_i(t)$, we have 
$$e_i(t) \in span \{ \gamma^{(1)}(t), \dots,  \gamma^{(i)}(t)\} \hskip .1in {\rm for} \hskip .1in i=1, \dots, n-1.$$
Thus,
$$e^\prime_i(t) \in span \{ \gamma^{(1)}(t), \dots,  \gamma^{(i+1)}(t)\} = span \{e_1(t), \dots, e_{i+1}(t)\}.$$
As a consequence, $$<e_i^\prime(t),e_j(t)>\ =\ 0 \hskip .1in {\rm whenever} \hskip .1in j\geq i+2.$$
If we differentiate the expression $<e_i(t),e_i(t)> \  =\ 1$ then we obtain $$<e_i^\prime(t),e_i(t)>+<e_i(t),e_i^\prime(t)> \ =\ 0,$$ from which we can conclude that $$<e_i^\prime(t),e_i(t)>\  =\ 0 \hskip .1in {\rm for} \hskip .1in 1\leq i \leq n.$$
In a similar manner, if $i\neq j$ then we differentiate the expression $<e_i(t),e_j(t)> \  =\ 0$ to obtain $$<e_i^\prime(t),e_j(t)>+<e_i(t),e_j^\prime(t)> \ =\ 0,$$ from which we can conclude that $$<e_i^\prime(t),e_j(t)>\  =\ -<e_j^\prime(t),e_i(t)>. $$
Let $E$ denote the orthonormal matrix whose columns are $e_1(t), \dots, e_n(t)$. The above formulas show that $E^TE^\prime =K$ with
$K$ a tri diagonal skew symmetric matrix.
Since $E$ is orthonormal (thus $EE^T=I$), we can multiply on the left by $E$ to arrive at the expression $E^\prime = EK$. Recalling that $\kappa_i(t) =\ <e_i^\prime(t),e_{i+1}(t)>$ we can express $K$ as:

$$K = 
 \left( \begin{array}{cccccc}
0 & -\kappa_1(t) & 0 & 0 & 0  \\
\kappa_1(t)      & 0 & -\kappa_2(t) & 0 & 0  \\
0 & \kappa_2(t) & 0 & \ddots & 0  \\
0 &  0& \ddots & 0 & -\kappa_{n-1}(t)\\
0 & 0 & 0 &  \kappa_{n-1}(t) & 0  \end{array} \right)$$

If the generalized curvature functions $\kappa_1(t), \dots, \kappa_{n-1}(t)$ in the matrix $K$ are constant, then the solution to the differential equation, $E^\prime = EK$, can be shown to be (up to translation) of the form
\begin{equation}
\label{FS-even}
\gamma_e(t) = \bmatrix a_1\cos(\alpha_1 t)\\ a_1 \sin( \alpha_1 t)\\  \vdots \\ a_{k}\cos(\alpha_k t) \\ a_k \sin( \alpha_k t))\endbmatrix \hskip .3in {\rm or} \hskip .3in \gamma_o(t) = \bmatrix a_1\cos(\alpha_1 t)\\ a_1 \sin( \alpha_1 t)\\  \vdots \\ a_{k}\cos(\alpha_k t) \\ a_k \sin( \alpha_k t)) \\ bt \endbmatrix \end{equation}
with the first equation, $\gamma_e(t)$, covering the case when $n$ is even with $k=n/2$ and the second equation covering the case when $n$ is odd with $k=(n-1)/2$ \cite{kuhnel2006differential}.

\section{Local approximation}
Consider a curve $\gamma(t)$ in $\mathbb R^n$. Recall that if $\gamma(t)$ is parameterized by arc length then $\gamma(t)$ is a solution to the differential equation $E^\prime = EK$.  We would like to understand the associated frame $e_1(t), \dots, e_n(t)$ and curvature functions $\kappa_1(t), \dots , \kappa_{n-1}(t)$
from a different point of view.   Specifically, consider points on the curve within an $\epsilon$-ball centered at a point
$s_0=\gamma(t_0)$. The tangent line at $s_0$ is approximated by taking the span of two points on $\gamma(t)$ in an $\epsilon$-ball centered at $s_0$ while the {\it osculating} plane at $s_0$ is approximated by 
taking the span of three points on $\gamma(t)$ in an $\epsilon$-ball centered at $s_0$. However, points on the curve in a small $\epsilon$-ball are nearly linear. The value of $\kappa_1(t_0)$ can be seen as a measure of the failure of the linearity of such points. In a similar manner, the value of the second curvature function, $\kappa_2(t_0)$ is a measure of the failure of planarity of points in an $\epsilon$-ball on the curve. This point of view will be considered more closely in the next section through the local singular value decomposition. In order to make this connection, it is helpful to replace the curve with an idealized version which agrees, to high order, with the curve at $\gamma(t_0)$.
 
 \subsection{Local approximation of curves in $\mathbb R^3$ and $\mathbb R^4$}
Consider a curve $\gamma(t)$ in $\mathbb R^3$. The helix of best fit to $\gamma$ at $\gamma(t_0)$ is the solution to the differential equation $E^\prime = EK_{t_0}$ where $K_{t_0}$ denotes the curvature matrix $K$ evaluated at $t_0$. Thus the curvature functions for the helix will be constants $\kappa_1=\kappa_1(t_0)$ and $\kappa_2=\kappa_2(t_0)$. In $\mathbb R^3$, the general solution, $g(t)$, to the differential equation, $E^\prime = EK_{t_0}$, has the form 
$$g(t) = (a\cos(\alpha t), a \sin(\alpha t), bt)+Constant.$$ The helix of best fit to $\gamma(t)$ at $\gamma(t_0)$ is given by
$$h(t) = g(t) - g(t_0)+\gamma(t_0).$$
If $||\gamma^{(1)}(t_0)||=1$ then we get the condition that
$$
a^2\alpha^2+b^2 = 1.
$$
The relationship between the curvature functions of the helix and the parameters $a,b,\alpha$ is:
\begin{align*} \kappa_1^2 & = a^2 \alpha^4,\\
\kappa_2^2 & = b^2 \alpha^2.\end{align*}

In a similar manner, if we solve the differential equation $E^\prime = EK_{t_0}$ for a curve $\gamma(t)$ in $\mathbb R^4$ then we obtain a toroidal curve of best fit at $\gamma(t_0)$ of the form
$$h(t) = g(t) - g(t_0)+\gamma(t_0),$$
where
$$g(t) = (a\cos(\alpha t), a \sin(\alpha t), b\cos(\beta t), b \sin(\beta t))+Constant.$$ 
We can relate $a,b,\alpha, \beta$ to the curvature functions as
\begin{align*} \kappa_1^2 &= a^2\alpha^4 + b^2\beta^4,\\
\kappa_1^2 \kappa_2^2 &= a^2\alpha^6 + b^2\beta^6 -\kappa_1^4,\\
\kappa_1^2 \kappa_2^2 \kappa^3_3 &= a^2\alpha^8 + b^2\beta^8 -\kappa_1^2(\kappa_1^2 +\kappa_2^2)^2,\end{align*}
where again we have assumed that the curve is parameterized by
arc length so
$$
a^2\alpha^2+b^2\beta^2 = 1.
$$
These equations are derived for $\kappa_1, \kappa_2, \kappa_3$ in \cite{kuhnel2006differential}.
Next we give the corresponding equations for curves in $\mathbb R^5$ and $\mathbb R^6$. The derivation is straightforward but tedious.

\subsection{Curvature relations in $\mathbb R^5$ and $\mathbb R^6$}
If we solve the differential equation $E^\prime = EK_{t_0}$ for a curve $\gamma(t)$ in $\mathbb R^5$ then we obtain a curve of best fit at $\gamma(t_0)$ of the form
$$h(t) = g(t) - g(t_0)+\gamma(t_0),$$ where
$$g(t) = (a\cos(\alpha t), a \sin(\alpha t), b\cos(\beta t), b \sin(\beta t),ct)+Constant.$$ 
We can relate $a,b,c,\alpha, \beta$ to the curvature functions as
$$\begin{array}{ccl}
1&=&{a}^{2}{\alpha}^{2}+{b}^{2}{\beta}^{2}+{c}^{2}\\
{{\kappa_1}}^{2}&=&{a}^{2}{\alpha}^{4}+{b}^{2}{\beta}^{4}\\
{{\kappa_1}}^{2}{{\it \kappa_2}}^{2}&=&{a}^{2}{\alpha}^{6}+{b}^{2}{\beta}^
{6}-{{\kappa_1}}^{4}\\
{{\kappa_1}}^{2}{{\it \kappa_2}}^{2}{{\it \kappa_3}}^{2}&=& {a}^{2}{\alpha}^{8}+{b}^{2}{\beta}^{8}-\kappa_1^2(\kappa_1^2 +\kappa_2^2)^2\\
{{\kappa_1}}^{2}{{\it \kappa_2}}^{2}{{\it \kappa_3}}^{2}{{\it \kappa_4}}^{2}&=&{a}^{2}{\alpha}^{10}+{b}^{2}{\beta}^{10} -\kappa_1^2((\kappa_1^2+\kappa_2^2+\kappa_3^2)(\kappa_2^2+\kappa_3^2)+\kappa_2^2\kappa_3^4).
\end{array}$$
In $\mathbb R^6$ the curve of best fit has
$$g(t) = (a\cos(\alpha t), a \sin(\alpha t), b\cos(\beta t), b \sin(\beta t),c\cos(\delta t),c\sin(\delta t))+Constant.$$ 
Letting $G_k=a^2\alpha^k+b^2\beta^k+c^2\delta^k$, we can relate $a,b,c,\alpha, \beta,\delta$ to the curvature functions as

\vspace{10pt}
\hspace{15pt}
$\begin{array}{ccl}
1&=&G_2\\
\kappa_1^2 &=& G_4\\
{\kappa_{{1}}}^{2}{\kappa_{{2}}}^{2}&=&G_6-{\kappa_{{1}}
}^{4}\\
{\kappa_{{1}}}^{2}{\kappa_{{2}}}^{2}{\kappa_{{3}}}^{2}&=&
G_8
-\kappa_1^2(\kappa_1^2 +\kappa_2^2)^2\\
{\kappa_{{1}}}^{2}{\kappa_{{2}}}^{2}{\kappa_{{3}}}^{2}{\kappa_{{4}}}^{2}
&=&G_{10}-\kappa_1^2((\kappa_1^2+\kappa_2^2+\kappa_3^2)(\kappa_2^2+\kappa_3^2)+\kappa_2^2\kappa_3^4)\\
{\kappa_{{1}}}^{2}{\kappa_{{2}}}^{2}{\kappa_{{3}}}^{2}{\kappa_{{4}}}^{
2}{\kappa_{{5}}}^{2}
&=&G_{12}-G_{10}(\kappa_1^2+\kappa_2^2+\kappa_3^2+\kappa_4^2)+F_8(\kappa_1^2\kappa_3^2+\kappa_4^2\kappa_1^2+\kappa_4^2\kappa_2^2).
\end{array}$
\newline

\section{The Local Singular Value Decomposition}
Recall that at each point $\gamma(t)\in \gamma(I)$, the Frenet-Serret frame
is determined by applying the Gram-Schmidt process to the
vectors $\gamma^{(1)}(t), \gamma^{(2)}(t), \dots, \gamma^{(n)}(t)$ (where $\gamma^{(k)}(t)$ denotes the $k^{th}$ derivative of $\gamma$ evaluated at $t$). We denote this ordered orthonormal basis $e_1(t), \dots, e_n(t)$ and let $E$ denote the orthonormal matrix whose columns are the $e_i(t)$.
The main intuition behind a local singular value analysis is to exploit the idea
that the Frenet-Serret frame may be viewed as finding the
subspace of best fit at a point on the curve. We consider the canonical solution
of the Frenet-Serret formula where $\kappa_i$ is assumed to be constant, i.e., 
the solutions to $E^\prime = EK$ given by Equation (\ref{FS-even}) where $K$ is constant.  We use an integral
formulation of the singular value decomposition, often referred to as the Karhunen-Lo\`eve
transformation, at a given point on the curve.  We then use a Taylor series approximation for $\gamma(t)$ to determine particular
eigenvalues of the Karhunen-Lo\`eve
transformation in the $\epsilon$-ball.  These relationships can be combined with
the relationships between the curvature constants and the curve parameters
to determine a formula for computing $\kappa_i$ locally
from the singular values of the Karhunen-Lo\`eve
transformation.

\subsection{Formulation}
Broomhead et al showed  that the {\it local} singular value decomposition
could be used to compute the topological dimension of a manifold from sampled points lying on the manifold \cite{broomhead1991local}.  This 
provided a powerful tool for many applications that involved modeling data on manifolds.
The original setting of \cite{broomhead1991local} concerned the reconstruction of  a manifold, via Takens' theorem, from scalar valued time series statistics of a dynamical system on the manifold.  
 The local singular value decomposition is also useful for applying 
manifold learning algorithms for geometric data analysis, e.g., 
local models such as charts \cite{kirby_1996sub2}, or global models based on
Whitney's embedding theorem \cite{kirby_1998a}.  A more detailed discussion may be found in
\cite{kirby2000geometric, little2009estimation}.

Following \cite{broomhead1987topological,broomhead1991local},
the {\it mean centered} covariance matrix of $\gamma(t)$ at $t$ is the matrix $$\overline{C}_{\epsilon}(t)=\frac{1}{2\epsilon} \int_{t -\epsilon}^{t + \epsilon} 
(\gamma(s)-\overline{\gamma}_{\epsilon}(t)) (\gamma(s)-\overline{\gamma}_{\epsilon}(t))^Tds$$ where
$$\overline{\gamma}_{\epsilon}(t)= \frac{1}{2\epsilon} \int_{t -\epsilon}^{t + \epsilon} 
\gamma(s) \ ds.$$
However, we will consider the closely related {\it on the curve} covariance matrix
\begin{equation}
	C_{\epsilon}(t)=\frac{1}{2\epsilon} \int_{t -\epsilon}^{t + \epsilon} 
(\gamma(s)-\gamma(t)) (\gamma(s)-\gamma(t))^Tds.
\end{equation}
By the singular value decomposition, we have a factorization $$C_{\epsilon}(t)=U_{\epsilon}(t) \Sigma_{\epsilon}(t)U^T_{\epsilon}(t)$$ where we assume that the diagonal elements in $\Sigma_{\epsilon}(t)$ are in monotone decreasing order. We call the columns of $U_{\epsilon}(t)$ the singular vectors of $C_{\epsilon}(t)$. Note that such singular vectors are only defined up to a factor of $\pm 1$. Let $U(t)=\lim_{\epsilon \rightarrow 0} U_{\epsilon}(t)$. The columns of $U(t)$, written $u_1(t), \dots, u_n(t)$, are called the local singular vectors at $\gamma(t)$. In a similar manner, one can define the local singular vectors $\overline{u}_1(t), \dots, \overline{u}_n(t)$ at $\gamma(t)$ by considering the limiting behavior of the singular vectors in the singular value decomposition of $\overline{C}_{\epsilon}(t)$ as $\epsilon$ tends towards zero.
\begin{theorem} \label{big}
Let $\gamma: I \rightarrow \mathbb R^n$ be a parametric curve of class $C^{n+1}$, regular of order $n$. Let $e_1(t), \dots, e_n(t)$ denote the Frenet-Serret frame at $\gamma(t)$. Let $u_1(t), \dots, u_n(t)$ denote the local singular vectors at $\gamma(t)$. Then for $i=1, \dots, n$, $e_i(t)=\pm u_i(t)$.
\end{theorem}

\begin{proof}
Let $\Gamma(t)$ denote the matrix whose columns are $\gamma^{(1)}(t), \dots, \gamma^{(n)}(t)$. The Frenet-Serret frame, $e_1(t), \dots, e_n(t)$, is obtained by applying the Gram-Schmidt process to the columns of $\Gamma(t)$. Thus $e_i(t)$ is a unit vector orthogonal to the span of  $\gamma^{(1)}(t), \dots, \gamma^{(i-1)}(t)$ but lying within the span of $\gamma^{(1)}(t), \dots, \gamma^{(i)}(t)$. Let $\mathbf v$ be the $n \times 1$ vector whose $k^{th}$ component is $(s-t)^k/k!$. Then $\Gamma(t)\mathbf v$ is the $n^{th}$ order Taylor series expansion for $\gamma(s)-\gamma(t)$ at $t$. Replacing $\gamma(s)-\gamma(t)$ with its Taylor series expansion leads to the $n^{th}$ order approximation
$$C_{\epsilon}(t)=\frac{1}{2\epsilon} \int_{t -\epsilon}^{t + \epsilon} 
(\gamma(s)-\gamma(t)) (\gamma(s)-\gamma(t))^Tds\approx \frac{1}{2\epsilon} \int_{t -\epsilon}^{t + \epsilon} (\Gamma(t)\mathbf v)(\Gamma(t)\mathbf v)^T\ ds.$$
We rewrite this as $$ \Gamma(t)\ \frac{1}{2\epsilon} \int_{t -\epsilon}^{t + \epsilon} \mathbf v \mathbf v^T \ ds\  \Gamma(t)^T = \Gamma(t)\ \mathcal E\  \Gamma(t)^T.$$ By the definition of $\mathcal E$, we compute that $$\mathcal E_{i,j}=\frac {\epsilon^{i+j}}{i! j! (i+j+1)} \hskip .1in {\rm if} \ i+j \ {\rm is \ even} \hskip .1in {\rm and} \hskip .1in \mathcal E_{i,j}=0 \hskip .1in {\rm if} \ i+j \ {\rm is \ odd.}$$
 We can express $\Gamma(t)\ \mathcal E\  \Gamma(t)^T$ in terms of the columns of $\Gamma(t)$ and the entries of $\mathcal E$ as
 $$\frac{\epsilon^2}{3}(c_1c_1^T) + \frac{\epsilon^4}{5}(\frac{1}{6}c_1c_3^T+\frac{1}{4}c_2c_2^T+\frac{1}{6}c_3c_1^T)+ \dots + \frac{\epsilon^{2k}}{2k+1}\sum_{i=1}^{2k-1}\frac{1}{i!(2k-i)!}c_ic_{2k-i}^T+\dots$$
 where $c_i=\gamma^{(i)}(t)$. As $\epsilon$ tends towards zero, this expression behaves more and more like the rank one matrix $\frac{\epsilon^2}{3}c_1c_1^T$. Noting that $c_1=\gamma^{(1)}(t)$, thus is a multiple of $e_1(t)$, we get $u_1(t)=\pm e_1(t)$. Let $P_1=I-e_1(t)e_1(t)^T$. Pre and post multiplying $\Gamma(t)\ \mathcal E\  \Gamma(t)^T$ with $P_1$ deflates away all terms involving $c_1$. More precisely, $$P_1\ \Gamma(t)\ \mathcal E\  \Gamma(t)^T\ P_1=\frac{\epsilon^4}{5}(\frac{1}{4}P_1c_2c_2^TP_1)+ \dots + \frac{\epsilon^{2k}}{2k+1}\sum_{i=2}^{2k-2}\frac{1}{i!(2k-i)!}P_1c_ic_{2k-i}^TP_1+\dots$$
 As $\epsilon$ tends towards zero, this deflated matrix behaves more and more like the rank one matrix $\frac{\epsilon^4}{5}(\frac{1}{4}P_1c_2c_2^TP_1)$. Noting that $P_1c_2=P_1\gamma^{(2)}(t)$, we see that  $P_1c_2$ is orthogonal to $\gamma^{(1)}(t)$ and is in the span of $\gamma^{(1)}(t), \gamma^{(2)}(t)$ thus is a multiple of $e_2(t)$. This leads to $u_2(t)=\pm e_2(t)$. We now pre and post multiply $P_1\ \Gamma(t)\ \mathcal E\  \Gamma(t)^T\ P_1$ with $P_2=I-e_2(t)e_2(t)^T$. Note that since $e_1(t)$ is orthogonal to $e_2(t)$, we have  $P_2P_1=I-e_1(t)e_1(t)^T-e_2(t)e_2(t)^T$. As $\epsilon$ tends towards zero, this doubly deflated matrix behaves more and more like the rank one matrix $\frac{\epsilon^6}{7}(\frac{1}{36}P_2P_1c_3c_3^TP_1P_2)$. Noting that $P_2P_1c_3=P_2P_1\gamma^{(3)}(t)$, we see that $P_2P_1c_3$ is orthogonal to the span of $\gamma^{(1)}(t), \gamma^{(2)}(t)$ but in the span of $\gamma^{(1)}(t), \gamma^{(2)}(t), \gamma^{(3)}(t)$ thus is a multiple of $e_3(t)$. This leads to $u_3(t)=\pm e_3(t)$. Continuing to deflate away previously found singular vectors, we obtain the relationship $e_i(t)=\pm u_i(t)$ for all $i$. Note that for this to work, $\mathcal E_{i,i}$ must be non-zero and $P_iP_{i-1}\cdots P_1\gamma^{(i+1)}(t)$ must be non-zero for each $i$. These conditions are satisfied since $\mathcal E_{i,i}=\frac{\epsilon^{2i}}{(2i+1)i!i!}$ and $\gamma$ is regular of order $n$ thus $\gamma^{(1)}(t), \dots, \gamma^{(n)}(t)$ are linearly independent.
\end{proof}

The previous theorem considered the relationship between the local singular vectors of a curve and the Frenet-Serret frame of a curve. We now consider the relationship between the local singular values of a curve and values of the curvature functions. More precisely, in the singular value decomposition $$C_{\epsilon}(t)=U_{\epsilon}(t) \Sigma_{\epsilon}(t)U^T_{\epsilon}(t)$$ we considered the limiting behavior of $U_{\epsilon}(t)$, as $\epsilon$ tends towards zero, in order to obtain the local singular vectors. We now consider the limiting behavior of $\Sigma_{\epsilon}(t)$ as $\epsilon$ tends towards zero. Note that the entries of $\Sigma_{\epsilon}(t)$ are the eigenvalues of $C_{\epsilon}(t)$ and that they  tend towards zero as $\epsilon$ tends towards zero. Let $\lambda_{i,\epsilon}(t)$ denote the $i^{th}$ diagonal entry of $\Sigma_{\epsilon}(t)$. We show that for some constant $c_i$, we can write $$\lambda_{i,\epsilon}(t)=c_i\epsilon^{2i} + O \left( {\epsilon}^{2i+2} \right).$$
The local singular values of $\gamma(t)$ are then defined as $\sigma_i(t)=\sqrt{c_i}\epsilon^i$.

In Section 2, we have explicitly expressed the curvature, for curves with constant curvature functions, in terms of the parameters of the curves. We now express the leading terms of the eigenvalues $\lambda_{i,\epsilon}(t)$ in terms of the parameters of the curves. This allows us to derive a relationship of the form
\begin{equation}
\kappa^2_i(t)=a_i  \lim_{\epsilon\rightarrow 0}\frac{\lambda_{i+1,\epsilon}(t)}{\lambda_{1,\epsilon}(t)\lambda_{i,\epsilon}(t)},
\end{equation} where $a_i$ is a constant with known value.
From this we obtain $$\kappa_i(t)=\sqrt{a_i } \frac{\sigma_{i+1}(t)}{\sigma_1(t)\sigma_i(t)}.$$

\subsection{Two dimensions}

Consider a two dimensional curve with constant curvature $\kappa_1=1/a$. This will be a circle of radius $a$. Up to translation, its parameterized form is 
 $\gamma(s) = (a\cos( \alpha s), a\sin(\alpha s))$. If we assume that the circle is parameterized by arc length then we obtain the constraint $a^2\alpha^2=1$.
The components of the covariance matrix $C_{\epsilon}(0)$ are:
$$C_{11}= \frac{1}{2\epsilon } \int_{-\epsilon}^{\epsilon} 
(a\cos(\alpha s)-a)^2  ds,$$
$$C_{22} = \frac{1}{2\epsilon} \int_{-\epsilon}^{\epsilon} a^2 \sin^2(\alpha s) ds,$$
with 
$$C_{12} = C_{21} = \frac{1}{2\epsilon} \int_{-\epsilon}^{\epsilon}  (a\cos(\alpha s)-a) \sin(s) ds = 0$$
since the integrand is an odd function.

We follow the usual convention of ordering the eigenvalues by decreasing magnitude so
\begin{align*}
&\lambda_{1,\epsilon}(0) =\frac{1}{3}  a^2 \alpha^2 \epsilon^2+O(\epsilon^4),\\
&\lambda_{2,\epsilon}(0) =\frac{1}{20}a^2 \alpha^4 \epsilon^4+O(\epsilon^6),\\
&\lim_{\epsilon \rightarrow 0} \frac{\lambda_{2,\epsilon}(0)}{\lambda^2_{1,\epsilon}(0)} = 
{\frac{9}{20 a^2}}.
\end{align*}
Given that the curvature
$\kappa_1 = 1/a$,
 we obtain the following expression for $\kappa_1$ in terms of the
local singular values of the circle:
$$
\label{k1_ev}
\kappa_1 = \sqrt{\frac{20}{9} }  \frac{\sigma_2}{\sigma_1^2} =\frac{\sqrt{20}}{3} \frac{\sigma_2}{\sigma_1^2}.
$$

\subsection{Three and Four dimensions}
Here we consider curves in $\mathbb R^3$ with constant $\kappa_1,\kappa_2$. Up to translation, such a curve will have the form 
$$\gamma(s) = (a\cos(\alpha s), a \sin( \alpha s), bs).$$

Assuming the curve is parameterized by arc length we have $a^2\alpha^2 +b^2=1$. The covariance matrix, $C_{\epsilon}(t)$, is a $3 \times 3$ matrix with eigenvalues
\begin{align*}
\lambda_1 &= \frac{1}{3} \epsilon^2 +O(\epsilon^4)\\
\lambda_2 &= \frac{1}{20}a^2\alpha^4 \epsilon^4  +O(\epsilon^6)\\
\lambda_3 &= \frac{1}{1575} a^2\alpha^6b^2\epsilon^6+O(\epsilon^{8})
\end{align*}

Recalling from Section 2 the equations for $\kappa_1,\kappa_2$ in terms of the parameters $a,\alpha,b$, we obtain
$$
\label{k2_ev}
\kappa_1^2 = \frac{20}{9}  \lim_{\epsilon \rightarrow 0} \frac{\lambda_{2,\epsilon}(t)}{\lambda^2_{1,\epsilon}(t)}, \hskip .5in \kappa_2^2  =  \frac{105}{4}\lim_{\epsilon \rightarrow 0} \frac{\lambda_{3,\epsilon}(t)}{\lambda_{1,\epsilon}(t) \lambda_{2,\epsilon}(t)}.
$$
This leads to the expression of $\kappa_1,\kappa_2$ in terms of the singular values as:

$$\kappa_1=\frac{\sqrt{20}}{3} \frac{\sigma_2}{\sigma_1^2}  \hskip .5in {\rm and} \hskip .5in \kappa_2=\frac{\sqrt{105}}{2} \frac{\sigma_3}{\sigma_1 \sigma_2}.$$

Similarly for curves in $\mathbb R^4$, using elimination theory we establish the following representations of the $\kappa_i$ in terms of the local singular values:

$$\kappa_1=\frac{\sqrt{20}}{3} \frac{\sigma_2}{\sigma_1^2} , \hskip .1in \kappa_2=\frac{\sqrt{105}}{2} \frac{\sigma_3}{\sigma_1 \sigma_2}, \hskip .1in \kappa_3=\frac{\sqrt{336}}{5} \frac{\sigma_4}{\sigma_1 \sigma_3}. $$

\subsection{Higher dimensions}\label{patterns}

Given that many of the entries of $C_{\epsilon}(0)$ are odd functions,
the covariance matrix has a special structure with many zero entries.
For instance, the structure of the covariance matrix for  $n=6$ is
 
$$ \left[ \begin {array}{cccccc} C_{11}&0& C_{13}&0&C_{15}&0
\\ \noalign{\medskip}0&C_{22}&0&C_{24}&0&C_{26}
\\ \noalign{\medskip}C_{31}&0&C_{33}&0&C_{35}&0
\\ \noalign{\medskip}0&C_{42}&0&C_{44}&0&C_{46}
\\ \noalign{\medskip}C_{51}&0&C_{53}&0&C_{55}&0
\\ \noalign{\medskip}0&C_{62}&0&C_{64}&0&C_{66}\end {array}
 \right] $$

We can permute the columns and rows of this matrix
an even number of times to obtain the block matrix
$$
\left[ \begin {array}{cccccc}  C_{11}&C_{13}&C_{15} & 0 & 0&0
\\ \noalign{\medskip}C_{31}&C_{33}&C_{35} & 0 & 0&0
\\ \noalign{\medskip}C_{51}&C_{53}&C_{55} & 0 & 0&0
\\ \noalign{\medskip}0&0&0&C_{22}&C_{24}&C_{26}
\\ \noalign{\medskip}0&0&0&C_{24}&C_{44}&C_{46}
\\ \noalign{\medskip}0&0&0&C_{26}&{ C_{46}}&C_{66}\end {array}
 \right] $$

Thus we observe the more computationally efficient approach to computing the
eigenvalues by computing the eigenvalues of the block submatrices.

For curves in $\mathbb R^5$ we obtain:
\begin{equation}
	\kappa_1=\frac{\sqrt{20}}{3} \frac{\sigma_2}{\sigma_1^2} , \hskip .1in \kappa_2=\frac{\sqrt{105}}{2} \frac{\sigma_3}{\sigma_1 \sigma_2}, \hskip .1in \kappa_3=\frac{\sqrt{336}}{5} \frac{\sigma_4}{\sigma_1 \sigma_3}, \hskip .1in \kappa_4=\dfrac{\sqrt {825}}{4} \frac{\sigma_5}{\sigma_1 \sigma_4}.\end{equation}
And for curves in $\mathbb R^6$ the same expressions for $\kappa_1, \kappa_2, \kappa_s, \kappa_4$ hold plus the additional relationship \begin{equation}\kappa_5=\dfrac{\sqrt {1716}}{7} \frac{\sigma_6}{\sigma_1 \sigma_5}.\end{equation}

Throughout this section, we have assumed the curve to be parameterized with respect to arc length. The local computations can still be made without this assumption. What would change in the formulas in the previous section is that we would replace the assumption that $||\gamma^{(1)}(t_0)||=1$ with $||\gamma^{(1)}(t_0)||=r$. We obtain the same connection between the higher curvature functions and {\it ratios} of singular values. We summarize these results in the following theorem whose general proof for all dimensions is given in section \ref{proofbigt}:

\begin{theorem} \label{bigt}
Let $\gamma: I \rightarrow \mathbb R^n$ be a parametric curve of class $C^{n+1}$, regular of order $n$ for any $n\in\mathbb N$. Let $\kappa_j(t)$ denote the $j^{th}$ curvature function of $\gamma$ evaluated at $t$ and let $\sigma_j(t)$ denote the $j^{th}$ local singular value of $\gamma$ at $t$. For each $t\in I$ and each $j<n$,  \begin{equation}\kappa_j(t)=\sqrt{a_j}\frac{\sigma_{j+1}(t)}{\sigma_1(t) \sigma_j(t)}\hskip .1in {\rm with} \hskip .1in a_{j-1} = \left(\dfrac{j}{j+(-1)^j}\right)^2 {\dfrac{4j^2-1}{3}}.\end{equation}
\end{theorem}

\
 
The formula straightforwardly reproduces the results obtained above for the coefficients $$a_1=\frac{20}{9},\; a_2=\frac{105}{4},\; a_3=\frac{336}{25},\; a_4=\frac{825}{16},\; a_5=\frac{1716}{49}.$$ The proof of the general case requires the theory of Hankel determinants using orthogonal polynomials, which is reviewed in section \ref{gen_proof}.
Perhaps surprisingly, the numerator of this series arises in the number
of Kekul\'e structures in benzenoid hydrocarbons \cite{cyvin2013kekule} and the degrees of projections of rank loci \cite{aluffi2014degrees}.

%%%%%%%%%%%

\section{An example}
We consider the twisted cubic curve in $\mathbb R^3$ given parametrically as $\gamma(t)=[t,t^2,t^3]$. The Frenet-Serret frame can be shown to be:
\begin{equation*}
e_1(t) = \bmatrix \frac{1}{\sqrt{1+4t^2+9t^4}}\\ \frac{2t}{\sqrt{1+4t^2+9t^4}}\\ \frac{3t^2}{\sqrt{1+4t^2+9t^4}}\endbmatrix \hskip .2in e_2(t) = \bmatrix \frac{t(2+9t^2)}{\sqrt{1+4t^2+9t^4}\sqrt{1+9t^2+9t^4}}\\ \frac{1-9t^4}{\sqrt{1+4t^2+9t^4}\sqrt{1+9t^2+9t^4}}\\ \frac{3t+6t^3}{\sqrt{1+4t^2+9t^4}\sqrt{1+9t^2+9t^4}}\endbmatrix  \hskip .2in e_3(t) = \bmatrix \frac{3t^2}{\sqrt{1+9t^2+9t^4}}\\ \frac{-3t}{\sqrt{1+9t^2+9t^4}}\\ \frac{1}{\sqrt{1+9t^2+9t^4}}\endbmatrix\end{equation*}
while the functions $\kappa_1(t),\kappa_2(t)$ can be shown to be
$$\kappa_1(t)=\frac{2\sqrt{1+9t^2+9t^4}}{(1+4t^2+9t^4)^{3/2}}, \hskip .3in \kappa_2(t)=\frac{3}{1+9t^2+9t^4}.$$

Let $\epsilon = .001$ and let $t=3$. If we consider the singular value decomposition $C_{\epsilon}(t)=U_{\epsilon}(t) \Sigma_{\epsilon}(t)U^T_{\epsilon}(t)$ for $\gamma(t)$ then we can consider the singular vectors of $C_{\epsilon}(t)$ as a proxy for the local singular vectors of $\gamma(t)$ at $t=3$ and compare to the exact value for $e_i(t)$ at $t=3$. For instance, comparing the first singular vector to the first frame vector, we get
\begin{equation*}
u_{1,\epsilon}(3)=\bmatrix .036131465\\.216788800\\ .975549656\endbmatrix \hskip .2in e_1(3)=\bmatrix .036131468\\.216788812\\ .975549654\endbmatrix .
\end{equation*}
The other singular vectors, $u_{2,\epsilon}(3),u_{3,\epsilon}(3)$ are similarly close to $e_2(3),e_3(3)$. If we consider $$\sqrt{a_i}\frac{\sqrt{\lambda_{i+1,\epsilon}(t)}}{\sqrt{\lambda_{1,\epsilon}(t)}\sqrt{\lambda_{i,\epsilon}(t)}} \hskip .1in {\rm as \ a \ proxy\  for } \hskip .1in \kappa_i=\sqrt{a_i}\frac{\sigma_{i+1}(t)}{\sigma_1(t) \sigma_i(t)}$$ then we obtain the following estimates:
$$\kappa_1(3)\approx .0026865640, \ \ \kappa_2(3)\approx .0036991369,$$ whereas using the exact formulas, we can compare these values to
$$\kappa_1(3)= .0026865644..., \ \ \kappa_2(3) = .0036991368...$$
For these approximations, we used $\epsilon=10^{-3}$. With a choice of $\epsilon=10^{-6}$, for this example, we observed about 13 digits of accuracy. This example illustrates how the theorems of the previous section can be used to obtain very good approximations of both the Frenet-Serret frame and values of the curvature functions by considering small values of $\epsilon$.

%%%%%%%%%%%%%%%%%%%%%%%%%%%%%%%%%%%%%%%%%%%%%%%%%%%%%%%%%%%%%%%%%%%%%%
%%%%%%%%%%%%%%%%%%%%%%%%%%%%%%%%%%%%%%%%%%%%%%%%%%%%%%%%%%%%%%%%%%%%%%

\section{Hankel Matrices and Orthogonal Polynomials}\label{gen_proof}

After the previous explicit examples were worked out, we conjectured the formula of \ref{bigt} for $a_j$ and numerically verified the result for $\kappa_6, \kappa_7, \kappa_8$ by generating curves with prescribed non-constant curvature and solving the system $E^\prime = EK$ numerically; then, the local singular values were numerically approximated from the numerically
generated curves. The general proof is based on the following key result by F.J. Solis \cite{solis} for the expansions to leading order of the singular values:
\begin{lemma}
	Let $\gamma(t)$ be a regular curve in $\mathbb R^n$, and let $P_0$ be a point on the curve, then the eigenvalues associated with $C_\epsilon$ at $P_0$ are given by 
\begin{align*}
	& \lambda_1^\epsilon = p_1\epsilon^2 + O(\epsilon^4), \\
	& \lambda_j^\epsilon = \frac{(\kappa_1\cdots\kappa_{j-1} )^2}{(j!)^2}p_j\epsilon^{2j} + O(\epsilon^{2j+2}),\quad\quad j=2,\dots, n
\end{align*}
and the eigenvectors are given by the Frenet frame at $P_0$. The $\kappa_i$'s
 are the higher curvatures of the curve and $p_k$ is the k-th $(k=1,\dots,n)$ pivot of the $n\times n$ matrix $A_n$ defined by
 $$
A_{ij} = \begin{cases} 
      \frac{1}{i+j+1}, & \text{if } i+j\text{ is even}; \\
      0 & \text{otherwise}.
   \end{cases}
 $$
\end{lemma}

From his proof, a small typo is corrected for the denominator of $\lambda^\epsilon_j$ in the final statement. With this result we can express the curvatures $\kappa_j$ in terms of the singular values by expressing the  pivots as quotients of the determinants $B_j$ of $A_j$, that is $p_j=B_j/B_{j-1}$, so that:
\begin{equation}\label{solEq}
	\lim_{\epsilon\to 0}\frac{\lambda^\epsilon_{j+1}}{\lambda^\epsilon_1\lambda^\epsilon_j}=\kappa_j^2\frac{B_{j+1}B_{j-1}}{(j+1)^2B_1B_j^2}.
\end{equation}
The determinants $B_j$ are of Hankel type for the sequence $\{\mu_n\}_{n=0}^\infty=\{\frac{1}{3},0,\frac{1}{5},0,\frac{1}{7},...\}$
\begin{align*}
B_1 &= \frac{1}{3},\; B_2 = \left|\begin{array}{cc} \frac{1}{3} & 0 \\ 0 & \frac{1}{5}\end{array}\right| ,\; B_3 = \left|\begin{array}{ccc} \frac{1}{3} & 0 & \frac{1}{5} \\ 0 & \frac{1}{5} & 0 \\ \frac{1}{5} & 0 & \frac{1}{7} \end{array}\right|,\; B_j = \left|\begin{array}{ccccc} \mu_0 & \mu_1 & \mu_2 & \cdots & \mu_{j-1} \\ \mu_1 & \mu_2 & \mu_3 & \cdots & \mu_{j} \\ \mu_2 & \mu_3 & \mu_4 & \cdots & \mu_{j+1} \\ \vdots & \vdots & \vdots & & \vdots \\ \mu_{j-1} & \mu_j & \mu_{j+1} & \cdots & \mu_{2j-2} \end{array}\right|.
\end{align*}
Then to get our coefficient in \ref{bigt} amounts to showing that the aforementioned Hankel determinants satisfy the following recurrence relation:
\begin{equation}
\frac{B_jB_{j-2}}{(B_{j-1})^2} = \frac{(j+(-1)^j)^2}{4j^2-1}.
\end{equation}

This is indeed the case after we realize that such a recurrence relation appears in the theory of monic orthogonal polynomials generated from $\{x^n\}_{n=0}^\infty$ by Gram-Schmidt orthogonalization with respect to a measure giving our sequence $\mu_n$ as the integral moments. Indeed, choose a nondecreasing function $\lambda(x)$ on $\mathbb R$ having finite limits at $\pm\infty$ such that it induces a positive measure $d\lambda$ with finite moments to all orders
$$
\mu_n(d\lambda)= \int_{\mathbb R} x^n d\lambda(x),\quad n=0,1,2,...
$$
then apply the Gram-Schmidt orthogonalization procedure to $\{x^n\}_{n=0}^\infty$ using the scalar product 
$$
\langle p(x),q(x)\rangle = \int_{\mathbb R} p(x)q(x)d\lambda(x)
$$
to obtain a sequence of monic orthogonal polynomials $P_n(x)$ (without normalization). If the given scalar product is positive-definite, such a sequence is infinite and unique, and this is the case if $B_n>0$ for all $n\in\mathbb N$, see Gautschi \cite[th. 1.2, 1.6]{gautschi}. Moreover, in this case, the infinite sequence of monic orthogonal polynomials obtained in this manner obeys the recursion relation \cite[th. 1.27]{gautschi}:
\begin{equation}\label{recur}
P_{-1}(x)=0,\, P_{0}(x)=1,\;\quad P_{n+1}(x)=(x-\alpha_n)P_{n}(x) - \beta_n P_{n-1}(x)
\end{equation}
where $$\alpha_n=\frac{\langle P_n, xP_n \rangle}{\langle P_n, P_n\rangle},\quad \beta_n=\frac{\langle P_n, P_n \rangle}{\langle P_{n-1}, P_{n-1}\rangle}=\frac{||P_n(x)||^2}{||P_{n-1}(x)||^2}, \text{ for } n=1,2,\dots$$
The importance of this result is that the recursion coefficients $\beta_n$ are precisely the recursion coefficients of the Hankel determinants $B_n$ for the sequence $\mu_n$, as it is proved in \cite[eq. 2.1.5]{gautschi} 
	
\begin{equation}\label{recursion}
\beta_{j-1} = \frac{B_jB_{j-2}}{(B_{j-1})^2},\;\text{ for } n=2,3,\dots
\end{equation}
so finding a measure to reproduce our sequence as its moments and a way to compute the norms of the corresponding polynomials will yield our coefficient formula.
There is a fundamental determinantal representation of the monic orthogonal polynomials generated in the previous way \cite[th. 2.1]{gautschi}
$$
P_n(x)=\frac{1}{B_n}\left|\begin{matrix} \mu_0 & \mu_1 & \dots & \mu_{n} \\ \mu_1 & \mu_2 & \dots & \mu_{n+1} \\ \vdots & \vdots & & \vdots \\ \mu_{n-1} & \mu_{n} & \dots & \mu_{2n-1} \\ 1 & x & \vdots & x^n \end{matrix}\right|, \quad ||P_n(x)||^2= \frac{B_{n+1}}{B_{n}},
$$
that yields Heine's integral representation formula \cite[p. 288]{heine} by essentially pulling the integrals of each moment out of the determinant and expanding  
$$
P_n(x)=\frac{1}{n! B_n}\int\cdots\int_{\mathbb R^n}\prod_{i=1}^n (x-x_i)\prod_{1\leq l<k\leq n}(x_k-x_l)^2 d\lambda(x_1)\cdots d\lambda(x_n).
$$
Since the polynomials are monic, $B_n$ can be solved equating to $1$ the leading coefficient of the previous equation
\begin{equation}\label{heines}
B_n=\frac{1}{n!}\int_{\mathbb R^n}\prod_{1\leq l<k\leq n}(x_k-x_l)^2 d\lambda(x_1)\cdots d\lambda(x_n)
\end{equation}
which is a closed formula for all Hankel determinants of any sequence as long as this can be written as moments of a positive measure.

\section{Proof of Theorem \ref{bigt}}\label{proofbigt}

Using the theory above for Hankel determinants of particular type we arrive at the following key result.
\begin{theorem}
	For any inverse arithmetic sequence $\displaystyle \left\{\frac{1}{\alpha n+\beta}\right\}_{n=0}^\infty$, where $\alpha,\beta\in\mathbb{R}_{>0}$, the corresponding Hankel determinants
\begin{equation}
F_n(\alpha,\beta)=\left|\begin{array}{ccccc} \frac{1}{\beta} & \frac{1}{\alpha+\beta} & \frac{1}{2\alpha+\beta} & \cdots & \frac{1}{(n-1)\alpha+\beta} \\ \frac{1}{\alpha+\beta} & \frac{1}{2\alpha+\beta} & \frac{1}{3\alpha+\beta} & \cdots & \frac{1}{n\alpha+\beta} \\ \frac{1}{2\alpha+\beta} & \frac{1}{3\alpha+\beta} & \frac{1}{4\alpha+\beta} & \cdots & \frac{1}{(n+1)\alpha+\beta} \\ \vdots & \vdots & \vdots & & \vdots \\ \frac{1}{(n-1)\alpha+\beta} & \frac{1}{n\alpha+\beta} & \frac{1}{(n+1)\alpha+\beta} & \cdots & \frac{1}{(2n-2)\alpha+\beta} \end{array}\right|
\end{equation}
are given by
\begin{equation}\label{blockFormula}
	F_n(\alpha,\beta)=\frac{1}{\alpha^n}\prod_{k=0}^{n-1}\frac{\Gamma(\beta/\alpha + k)(k!)^2}{\Gamma(\beta/\alpha +n+k)} = \frac{1}{\alpha^n}\prod_{k=0}^{n-1}(k!)^2\prod_{j=0}^{n-1}\frac{\alpha}{\alpha(k+j)+\beta},
\end{equation} 
and satisfy the recursion relation
\begin{equation}
	\frac{F_nF_{n-2}}{F_{n-1}^2} =
	\frac{\alpha^2\left(\alpha(n-2)+\beta\right)^2(n-1)^2}{\left(\alpha(2n-2)+\beta\right)\left(\alpha(2n-3)+\beta\right)^2\left(\alpha(2n-4)+\beta\right)} ,
\end{equation}
starting with $\displaystyle F_1=\frac{1}{\beta},\; F_2=\frac{\alpha^2}{\beta(2\alpha+\beta)(\alpha+\beta)^2}$.
\end{theorem}
\begin{proof}
	Choose the function $\lambda(x)=x^{\beta/\alpha}/\beta$ which is always nondecreasing in the interval $[0,1]$ for $\beta/\alpha > 0$, then the corresponding positive measure
$$
d\lambda(x)=\chi_{[0,1]}\frac{x^{\beta/\alpha-1}}{\alpha}dx,
$$
where $\chi_I$ is the characteristic function of a measurable set $I\subset\mathbb R$, yields moments
$$
\mu_n = \int_{\mathbb R}x^nd\lambda(x)=\frac{1}{\alpha}\int_0^1 x^{n+\frac{\beta}{\alpha}-1}dx=\frac{1}{\alpha}\left[\frac{x^{n+\frac{\beta}{\alpha}}}{n+\frac{\beta}{\alpha}}\right]_0^1=\frac{1}{\alpha n +\beta}.
$$
Notice that this solves the Stieltjes moment problem uniquely for these sequences because our measure is infinitely supported on $[0,\infty)$, and its moments satisfy Carleman's condition \cite[th. 1.10]{shohat}. From this, the necessary condition $F_n>0$ is guaranteed to hold for any dimension $n$ \cite[th. 1.2]{shohat}, so the induced inner product is positive definite and thus the sequence of monic orthogonal polynomials $P_n(x)$ is infinite and unique. Thus their recurrence relations (\ref{recur}) hold for any $n\in\mathbb N$, so we can compute the determinants $F_n(\alpha,\beta)$ of any dimension. This is done by computing equation (\ref{heines})
$$
F_n(\alpha,\beta)=\frac{1}{n!}\int_0^1\cdots\int_0^1\prod_{i=1}^n\frac{x_i^{\frac{\beta}{\alpha}-1}}{\alpha}\prod_{1\leq l<k\leq n}(x_k-x_l)^2 dx_1\cdots dx_n
$$
by means of Selberg's integral formula \cite[8.1.1]{andrews}, an extension of Euler's Beta function which has applications in different fields within mathematics and physics:
\begin{align*}
		 \int_{[0,1]^n}\prod_{i=0}^n x_i^{a-1} (1-x_i)^{b-1}\prod_{1\leq l<k\leq j}|x_k-x_l|^{2g} dx_1\cdots dx_n = \prod_{k=0}^{n-1}\frac{\Gamma(a+kg)\Gamma(b+kg)\Gamma(1+(k+1)g)}{\Gamma(a+b+(n+k-1)g)\Gamma(1+g)},
\end{align*}
when $\Re e(a)>0,\Re e(b)>0$ and $\Re e(g) > -\min\{1/n, \Re e(a)/(n-1), \Re e(b)/(n-1)\}$. These conditions are satisfied for our case $a=\beta/\alpha>0$, and $b=g=1$. Therefore by substitution of these values
$$
F_n(\alpha,\beta)=\frac{1}{n!\alpha^n}\prod_{k=0}^{n-1}\frac{\Gamma(\beta/\alpha + k)\Gamma(1+k)\Gamma(2+k)}{\Gamma(\beta/\alpha +n+k)\Gamma(2)}=\frac{1}{\alpha^n}\prod_{k=0}^{n-1}\frac{\Gamma(\beta/\alpha + k)(k!)^2}{\Gamma(\beta/\alpha +n+k)}, 
$$
where the Gamma functions can be simplified by the factorial property $\Gamma(z+1)=z\Gamma(z)$ to get a closed formula:
$$
F_n(\alpha,\beta)= \frac{1}{\alpha^n}\prod_{k=0}^{n-1}(k!)^2\prod_{j=0}^{n-1}\frac{\alpha}{\alpha(k+j)+\beta}.
$$
Finally, the recursion equation (\ref{recursion}) can be worked out by telescoping the products of Gamma functions:
\begin{align*}
& \frac{F_nF_{n-2}}{F_{n-1}^2} =  \frac{1}{\alpha^n}\prod_{k=0}^{n-1}\frac{\Gamma(\frac{\beta}{\alpha} + k)(k!)^2}{\Gamma(\frac{\beta}{\alpha} +n+k)}\cdot 
	\alpha^{n-1}\prod_{k=0}^{n-2}\frac{\Gamma(\frac{\beta}{\alpha} +n-1+k)}{\Gamma(\frac{\beta}{\alpha}  + k)(k!)^2}\cdot \\
	& \quad\quad\quad\quad\quad\alpha^{n-1}\prod_{k=0}^{n-2}\frac{\Gamma(\frac{\beta}{\alpha} +n-1+k)}{\Gamma(\frac{\beta}{\alpha}  + k)(k!)^2}\cdot 
	\frac{1}{\alpha^{n-2}}\prod_{k=0}^{n-3}\frac{\Gamma(\frac{\beta}{\alpha}  + k)(k!)^2}{\Gamma(\frac{\beta}{\alpha} +n-2+k)} = \\
%\end{align*}
%\begin{align*}
	& \text{\small $ = \frac{\Gamma(\frac{\beta}{\alpha}  +n-1)(n-1)!^2}{\Gamma(\frac{\beta}{\alpha}  +n-2)(n-2)!^2}\prod_{k=0}^{n-1}\frac{1}{(\frac{\beta}{\alpha}  +n-1+k)\Gamma(\frac{\beta}{\alpha}  +n-1+k)}\prod_{k=0}^{n-2}\Gamma\left(\frac{\beta}{\alpha}  +n-1+k\right)\cdot $} \\ 
	& \quad\prod_{k=0}^{n-2}\left(\frac{\beta}{\alpha}  +n-2+k\right)\Gamma\left(\frac{\beta}{\alpha}  +n-2+k\right)\prod_{k=0}^{n-3}\frac{1}{\Gamma\left(\frac{\beta}{\alpha}  +n-2+k\right)} = \\
\end{align*}
\begin{align*}
&	= \frac{(\frac{\beta}{\alpha} +n-2)(n-1)^2\Gamma(\frac{\beta}{\alpha} +2n-4)}{\Gamma(\frac{\beta}{\alpha} +2n-2)}\prod_{k=0}^{n-1}\frac{1}{\left(\frac{\beta}{\alpha} +n-1+k\right)}\prod_{k=0}^{n-2}\left(\frac{\beta}{\alpha} +n-2+k\right) \\
&	= \frac{(\frac{\beta}{\alpha}+n-2)^2(n-1)^2}{(\frac{\beta}{\alpha}+2n-2)(\frac{\beta}{\alpha}+2n-3)^2(\frac{\beta}{\alpha}+2n-4)}
\end{align*}
which yields the stated formula upon multiplying numerator and denominator by $\alpha^4$.
\end{proof}

Remarkably, this means that our polynomial recursion coefficients satisfy $\displaystyle\beta_{n} = \frac{1}{4}\beta^J_n$, where $\beta^J_n$ are those of the classical monic Jacobi polynomials of type $(\frac{\beta}{\alpha} -1, 0)$. These are generated by the measure $\chi_{[-1,1]}(1-x)^{\frac{\beta}{\alpha}-1}dx$, which induces a completely different moment sequence and set of orthogonal polynomials.

Our actual determinants $B_n$ have alternating 0's in the even positions of the moment sequence, so a block decomposition is needed to get them into the form of the theorem.

\begin{corollary}\label{hankeldet2}
	For any sequence of type $\displaystyle \left\{\frac{1}{\alpha n+\beta},0\right\}_{n=0}^\infty$ with $\alpha,\beta\in\mathbb R_{>0}$, where zeros alternate every other position, the corresponding Hankel determinants $B_n$ are given by the following block decomposition for even $n=2m$ or odd $n=2m-1$ dimension, $m\in\mathbb{N}$:
	$$
		B_{2m} = F_m(\alpha,\beta) F_m(\alpha,\beta+\alpha),\quad B_{2m-1} = F_m(\alpha,\beta) F_{m-1}(\alpha,\beta+\alpha),
	$$
	and obey the recurrence relations:
\begin{equation}\label{Heq1}
	\frac{B_{2m}B_{2m-2}}{(B_{2m-1})^2}=\frac{(\alpha(m-1)+\beta)^2}{(\alpha(2m-1)+\beta)(\alpha(2m-2)+\beta)},
\end{equation}
\begin{equation}\label{Heq2}
	\frac{B_{2m-1}B_{2m-3}}{(B_{2m-2})^2}=\frac{\alpha^2(m-1)^2}{(\alpha(2m-2)+\beta)(\alpha(2m-3)+\beta)},
\end{equation}
starting with $\displaystyle B_1=\frac{1}{\beta},\; B_2=\frac{1}{\beta(\alpha+\beta)}$.
\end{corollary}
\begin{proof}
	The Hankel determinants with 0's at every even position of the first row can be decomposed into blocks by the procedure mentioned in Section \ref{patterns} without altering the overall sign. Notice that the second block has as Hankel sequence the original one but shifted in index by $+1$, so the blocks are $F_m:=F_m(\alpha,\beta)$ and $E_m:=F_m(\alpha,\beta+\alpha)$. Analogously for $n=2m-1$, but in this case the number of 0's is now $m-1$, so the size of the second block is $(m-1)^2$ whereas the first is still $m^2$. Thus
	$$ B_{2m} = F_m E_m,\quad B_{2m-1} = F_m E_{m-1}.$$
Whence the recursion coefficients for the induced polynomials are, for even $n$,
		$$
		\beta_{n-1}=\beta_{2m-1}=\frac{B_{2m}B_{2(m-1)}}{B^2_{2m-1}}=\frac{E_m}{E_{m-1}}\frac{F_{m-1}}{F_m},
		$$
		and for odd $n$:
		$$
		\beta_{n-1}=\beta_{2m-2}=\frac{B_{2m-1}B_{2(m-1)-1}}{B^2_{2(m-1)}}=\frac{E_{m-2}}{E_{m-1}}\frac{F_{m}}{F_{m-1}}.
		$$
Therefore using (\ref{blockFormula}), that the corresponding $\beta/\alpha$ for the $E_m$ blocks is $\beta/\alpha +1$ and the factorial property of the Gamma function, the products can be simplified in the same way as in our previous proof:

\begin{align*}
	\frac{B_{2m}B_{2(m-1)}}{B^2_{2m-1}} & = \frac{1}{\alpha^m}\prod_{k=0}^{m-1}\frac{\Gamma(\beta/\alpha +1+ k)(k!)^2}{\Gamma(\beta/\alpha+1 +m+k)}\cdot 
	\alpha^{m-1}\prod_{k=0}^{m-2}\frac{\Gamma(\beta/\alpha+m+k)}{\Gamma(\beta/\alpha +1+ k)(k!)^2} \\
	& \frac{1}{\alpha^{m-1}}\prod_{k=0}^{m-2}\frac{\Gamma(\beta/\alpha + k)(k!)^2}{\Gamma(\beta/\alpha +m-1+k)}\cdot 
	\alpha^{m}\prod_{k=0}^{m-1}\frac{\Gamma(\beta/\alpha +m+k)}{\Gamma(\beta/\alpha + k)(k!)^2} = \\
	& = (\beta/\alpha+m-1)\prod_{k=0}^{m-1}\frac{1}{(\beta/\alpha+m+k)}\cdot \prod_{k=0}^{m-2}(\beta/\alpha+m+k-1) = \\
	& = \frac{(\beta/\alpha+m-1)^2}{(\beta/\alpha+2m-1)(\beta/\alpha+2m-2)}.
\end{align*}
Similarly,
\begin{align*}
	\frac{B_{2m-1}B_{2(m-1)-1}}{B^2_{2(m-1)}} & = \frac{1}{\alpha^{m-2}}\prod_{k=0}^{m-3}\frac{\Gamma(\beta/\alpha +1+ k)(k!)^2}{\Gamma(\beta/\alpha +m-1+k)}\cdot 
	\alpha^{m-1}\prod_{k=0}^{m-2}\frac{\Gamma(\beta/\alpha+m+k)}{\Gamma(\beta/\alpha +1+ k)(k!)^2} \\
	& \frac{1}{\alpha^{m}}\prod_{k=0}^{m-1}\frac{\Gamma(\beta/\alpha + k)(k!)^2}{\Gamma(\beta/\alpha +m+k)}\cdot 
	\alpha^{m-1}\prod_{k=0}^{m-2}\frac{\Gamma(\beta/\alpha +m-1+k)}{\Gamma(\beta/\alpha + k)(k!)^2} = \\
	& = \frac{(m-1)!^2\Gamma(\beta/\alpha+m-1)\Gamma(\beta/\alpha+2m-3)}{(m-2)!^2\Gamma(\beta/\alpha+m-1)\Gamma(\beta/\alpha+2m-1)} = \\
	& = \frac{(m-1)^2}{(\beta/\alpha+2m-2)(\beta/\alpha+2m-3)}.
\end{align*}

\end{proof}

Finally the coefficient formula of Section \ref{bigt} is obtained from this using (\ref{solEq}).
\begin{corollary}
	The Hankel determinants of size $n\times n$
$$
B_n=\det(A_n),\quad 
	(A_n)_{ij} = \begin{cases} 
      \frac{1}{i+j+1}, & \text{if } i+j\text{ is even}; \\
      0 & \text{otherwise}.
   \end{cases}
$$
satisfy the recurrence relation
\begin{equation}
	\frac{B_{n}B_{n-2}}{(B_{n-1})^2} = \frac{(n+(-1)^n)^2}{4n^2-1}.
\end{equation}
\end{corollary}
\begin{proof}
	Notice the matrix entry at $(A_n)_{ij}$ is precisely the element of the sequence $\displaystyle \left\{\frac{1}{2 n+3},0\right\}_{n=0}^\infty$ where $n=i+j-2$. Thus substituting $\alpha=2$ and $\beta=3$ into the equations (\ref{Heq1}), (\ref{Heq2}) above, the result follows straightforwardly when simplifying the theorem formulas after indices are written in terms of the dimension, $m=n/2$ or $m=(n+1)/2$ for the even and odd cases respectively.
\end{proof}

%%%%%%%%%%%%%%%%%%%%%%%%%%%%%%%%%%%%%%%%%%%%%%%%%%%%%%%%%%%%%%%%%%%%%%
%%%%%%%%%%%%%%%%%%%%%%%%%%%%%%%%%%%%%%%%%%%%%%%%%%%%%%%%%%%%%%%%%%%%%%

\section{Conclusions}
\label{sec:conclusions}

In this paper, we established the close connection between the Frenet-Serret apparatus and the local singular value decomposition of regular curves in $\mathbb R^n$. The local singular value decomposition was defined as the limit of the singular value decomposition of a family of covariance matrices defined on the curve.
In particular, we showed in Theorem~\ref{big} that the Frenet-Serret frame and the
local singular vectors of regular curves in $\mathbb R^n$ agree (up to a factor of $\pm 1$). In addition we showed in Theorem~\ref{bigt} that values of each of the curvature functions can be expressed in terms of ratios of local singular values for regular curves in $\mathbb R^n$ for any dimension, with a proportionality coefficient that was obtained exactly through its relation to Hankel determinants via monic orthogonal polynomials. With this, the techniques allow for highly accurate approximations of the Frenet-Serret apparatus in terms of local SVD computations.

%\appendix
%\section{An example appendix} 
%\lipsum[71]

\section*{Acknowledgments}
We would like to thank Louis Scharf and Olivier Pinaud for helpful discussions. 

\bibliographystyle{amsplain}
\bibliography{references}

\end{document}